\documentclass{daj}

\usepackage{amsmath,amsthm,amssymb}

\newtheorem{theorem}{Theorem}[section]
\newtheorem{lemma}[theorem]{Lemma}
\newtheorem{corollary}[theorem]{Corollary}

\theoremstyle{definition}
\newtheorem{example}[theorem]{Example}

\theoremstyle{remark}
\newtheorem*{remark}{Remark}

\newcommand{\ang}[1]{\left\langle #1 \right\rangle}

\newcommand{\wh}{\widehat}
\newcommand{\ol}{\overline}

\DeclareMathOperator{\Tr}{Tr}
\DeclareMathOperator{\Cay}{Cay}
\DeclareMathOperator{\Bip}{Bip}

\newcommand{\HS}{\mathrm{HS}}

\newcommand{\x}{\times}

\newcommand{\e}{\epsilon}

\newcommand{\cut}{\mathsf{cut}}
\newcommand{\G}{\mathsf{G}}
\newcommand{\EE}{\mathbb{E}}
\newcommand{\CC}{\mathbb{C}}
\newcommand{\DD}{\mathbb{D}}
\newcommand{\HH}{\mathbb{H}}
\newcommand{\RR}{\mathbb{R}}

\dajAUTHORdetails{%
  title = {Quasirandom Cayley graphs},
  author = {David Conlon and Yufei Zhao},
  plaintextauthor = {David Conlon, Yufei Zhao},
}   

\dajEDITORdetails{%
   year={2017},
   number={6},
   received={21 April 2016},   
   revised={28 February 2017},    
   published={8 March 2017},  
   doi={10.19086/da.1294},       
}   

\begin{document}

\begin{frontmatter}[classification=text]


\author[dc]{David Conlon\thanks{The first author was supported by a Royal Society University Research Fellowship and ERC Starting Grant 676632.}}
\author[yz]{Yufei Zhao\thanks{The second author was supported by an Esm\'ee Fairbairn Junior Research Fellowship at New College, Oxford.}}

\begin{abstract}
We prove that the properties of having small discrepancy and having small second eigenvalue are equivalent in Cayley graphs, extending a result of Kohayakawa, R\"odl, and Schacht, who treated the abelian case. The proof relies on Grothendieck's inequality. As a corollary, we also prove that a similar result holds in all vertex-transitive graphs.
\end{abstract}
\end{frontmatter}


\section{Introduction} \label{sec:intro}

A fundamental result of Chung, Graham, and Wilson~\cite{CGW89}, building on earlier work of Thomason~\cite{T87, T872}, states that for a sequence of graphs of density $p$, where $p$ is a fixed positive constant, a number of seemingly distinct notions of quasirandomness are equivalent. The following theorem details some of these equivalences.

\begin{theorem}[Chung--Graham--Wilson] \label{CGW}
For any fixed $0 < p < 1$ and any sequence of graphs $(\Gamma_n)_{n \in \mathbb{N}}$ with $|V(\Gamma_n)| = n$, the following properties are equivalent:
\begin{itemize}
\item[$P_1$:]
for all subsets $S,T \subseteq V(\Gamma_n)$, $e(S,T) = p |S||T| + o(pn^2)$;

\item[$P_2$:]
$e(\Gamma_n) \geq (1+o(1))p \binom{n}{2}$, $\lambda_1(\Gamma_n) = (1 + o(1)) p n$ and $|\lambda_2(\Gamma_n)| = o(pn)$, where $\lambda_i(\Gamma_n)$ is the $i$th largest eigenvalue, in absolute value, of the adjacency matrix of $\Gamma_n$;

\item[$P_3$:]
for all graphs $H$, the number of labeled copies of $H$ in $\Gamma_n$ is $(1 + o(1)) p^{e(H)} n^{v(H)}$.

\end{itemize}
\end{theorem}
\noindent
Here, the {\it adjacency matrix} $A(\Gamma)$ of an $n$-vertex graph $\Gamma$ is the $n \x n$ matrix $(a_{st})_{s, t \in [n]}$ where  $a_{st} = 1$ if $s$ and $t$ are adjacent in $\Gamma$ and $0$ otherwise.

In this paper, we will be concerned with studying the extent to which this theorem, and particularly the equivalence between the first two properties, extends to sparse graphs, that is, graphs where the density $p$ tends to zero. We will focus our attention on regular graphs. We then say that an $n$-vertex $d$-regular graph $\Gamma = (V,E)$ is {\it $\epsilon$-uniform} if, for all $S,T \subseteq V$,
\[\left| e(S,T) - \frac{d}{n}|S||T|\right| \le \epsilon dn.\]
The equivalence between properties $P_1$ and $P_2$ of Theorem~\ref{CGW} now implies that if $p$ is fixed and $\Gamma_n$ is a sequence of graphs with $|V(\Gamma_n)| = n$ and $\Gamma_n$ regular of degree $d_n = pn$, then the sequence $\Gamma_n$ is $o(1)$-uniform if and only if $|\lambda_2(\Gamma_n)| = o(d_n)$. 

One direction of this equivalence follows from the famous expander mixing lemma. This says that if $\Gamma = (V, E)$ is an {\it $(n, d, \lambda)$-graph}, that is, an $n$-vertex $d$-regular graph such that all eigenvalues of the adjacency matrix $A(\Gamma)$, except the largest, are bounded above in absolute value by $\lambda$, then  
\[\left| e(S,T) - \frac{d}{n}|S||T|\right| \le \lambda \sqrt{|S||T|}\]
for all $S,T \subseteq V$. Thus, if $|\lambda_2(\Gamma)| \leq \e d$, $\Gamma$ is $\e$-uniform. Note that this statement does not rely on knowing that $d$ is large, applying just as well when it is constant as when it is on the order of $n$. This observation is of fundamental importance in the study of expander graphs (see~\cite{HLW06}).

Answering a question of Chung and Graham~\cite{CG02}, Krivelevich and Sudakov~\cite{KS06} showed that the converse does not hold. That is, there are sequences of sparse graphs for which the discrepancy condition $P_1$ of Theorem~\ref{CGW} does not imply the eigenvalue condition $P_2$. However, their example is not regular (and a later example due to Bollob\'as and Nikiforov~\cite{BN04} is not connected), so we give an alternative example below.

\begin{example}\label{nonequiv}
Take three disjoint vertex sets $U$, $V$, and $W$, with $|U| = |V| = t$ and $|W| = n - 2t$. Choose a $d$-regular graph $\Gamma$ on vertex set $U \cup V \cup W$ such that every vertex in $U$ has no neighbours in $U$ and exactly $\lambda$ neighbours in $V$, and vice versa for $V$, while each vertex in $W$ has the same number of neighbours in $U$ and in $V$. By taking the vector $y = (y_1, \dots, y_n)$ with $y_j = 1$ if $j \in U$, $-1$ if $j \in V$, and 0 otherwise, we see that $-\lambda$ is an eigenvalue of the adjacency matrix $A(\Gamma)$. More concretely, suppose $d$ is even, let $t = \lambda = d/2$, and let the bipartite graph between $U$ and $V$ be complete. We split the remaining set of vertices $W$ into two subsets $W_0$ and $W_1$, where $|W_1| = d^2/4$ and each vertex in $W_1$ is joined to both endpoints of a unique associated edge in $\Gamma[U,V]$, while no vertex in $W_0$ is joined to $U \cup V$. In $W$, we place a random graph such that every vertex in $W_0$ has degree $d$ and every vertex in $W_1$ has degree $d-2$. The resulting graph is $d$-regular. Moreover, provided $d \rightarrow \infty$ and $d = o(n)$, a simple calculation (similar to the proof of Proposition 12 in~\cite{BN04}) implies that $G$ is $o(1)$-uniform, despite having $-d/2$ as an eigenvalue.
\end{example}

Given such examples, a recent result of Kohayakawa, R\"odl, and Schacht~\cite{KRS16} comes as something of a surprise. Suppose that $G$ is a finite group and $S$ is a subset of $G$. The (directed) {\it Cayley graph} $\Cay(G,S)$ is formed by taking the elements of $G$ as the vertex set and $\{(sg,g) : g \in G, s \in S\}$ as the set of (directed) edges. If $S$ is symmetric, that is, $S = S^{-1}$, the graph can be viewed as undirected. Unless stated otherwise, we will always use the term Cayley graph to refer to such an undirected graph with a symmetric generating set. We note that many of the standard examples of quasirandom graphs, including Paley graphs and the Ramanujan graphs of Lubotzky, Phillips, and Sarnak~\cite{LPS88} and Margulis~\cite{M88}, are Cayley graphs.

\begin{theorem}[Kohayakawa--R\"odl--Schacht] \label{KRS}
Let $G$ be an abelian group. Then every $\e$-uniform Cayley graph $\Cay(G, S)$ is an $(n,d,\lambda)$-graph with $n = |G|$, $d = |S|$, and $\lambda \le C\epsilon d$ for some absolute constant $C$.
\end{theorem}

The main result of this paper generalizes Theorem~\ref{KRS} to any finite group.

\begin{theorem}\label{thm:cayley}
Every $\e$-uniform Cayley graph $\Cay(G, S)$ is an $(n,d,\lambda)$-graph with $n = |G|$, $d = |S|$, and $\lambda \le 8\epsilon d$.
\end{theorem}

It is not hard to lift this result to all vertex-transitive graphs, where a graph is said to be {\it vertex transitive} if the automorphism group of the graph acts transitively on the vertices. Note that Cayley graphs are always vertex transitive, while the example of the Petersen graph shows that the converse is not true.

\begin{corollary} \label{transitive}
Every $n$-vertex $d$-regular $\epsilon$-uniform vertex-transitive graph is an $(n,d,\lambda)$-graph with $\lambda \le 8\epsilon d$.
\end{corollary}

Suppose now that $\Gamma$ is an $n$-vertex $d$-regular vertex-transitive graph. If $\Gamma$ is $\epsilon$-uniform, this means that
\begin{equation} \label{weakdisc}
\left| e(S,T) - \frac{d}{n}|S||T|\right| \le \epsilon dn
\end{equation}
for all $S, T \subseteq V(\Gamma)$. Corollary~\ref{transitive} then says that $\Gamma$ is an $(n,d,\lambda)$-graph with $\lambda \leq 8 \epsilon d$. In turn, the expander mixing lemma shows that
\begin{equation} \label{strongdisc}
\left| e(S,T) - \frac{d}{n}|S||T|\right| \le 8 \epsilon d \sqrt{|S||T|}
\end{equation}
for all $S,T \subseteq V(\Gamma)$. Thus, Corollary~\ref{transitive} may be seen as saying that in vertex-transitive graphs the discrepancy condition~\eqref{weakdisc} bootstraps itself to the stronger condition~\eqref{strongdisc}.

For comparison, we note also a result of Bilu and Linial~\cite{BL06} which says that if $\Gamma$ is an $n$-vertex $d$-regular graph satisfying the stronger discrepancy condition~\eqref{strongdisc}, then it is an $(n, d, \lambda)$-graph with $\lambda = O(\epsilon d \log(2/\epsilon))$. Our result shows that in Cayley graphs, one can derive the stronger conclusion $\lambda = O(\epsilon d)$ from the weaker condition \eqref{weakdisc}.

We begin the process of proving Theorem~\ref{thm:cayley} and Corollary~\ref{transitive} in the next section by introducing two matrix norms, the cut norm and the spectral norm, which capture the properties of having small discrepancy and having small second eigenvalue. The rough idea then will be to show that if $A$ is the adjacency matrix of a vertex-transitive graph $\Gamma$, the spectral norm of $A$ agrees with a semidefinite relaxation of the cut norm. 

We note that some similar ideas have been used in this context before. In particular, a semidefinite relaxation was used by Alon et al.~\cite{ACHKRS10} to prove the related result that if a graph $G$ has small discrepancy, then one can remove a small fraction of the vertices so that in the remaining graph all eigenvalues, except the largest, are small.

\section{Norms and relaxations} \label{sec:relax}

For any vector $x = (x_1, \dots, x_n) \in \RR^n$, we define its $p$-norm, $1 \le p \le \infty$, by
\[
|x|_p := (|x_1|^p + \dots + |x_n|^p)^{1/p}
\]
and
\[
|x|_\infty := \max \{ |x_1|, \dots, |x_n|\}.
\]
In particular, we write $|x| := |x|_2$ for the Euclidean norm. The inner product on $\RR^n$ is
\[
\ang{x,y} := x^* y = x_1 y_1 + \cdots + x_n y_n.
\]

Let $A = (a_{st})_{s \in [m], t \in [n]} \in \RR^{m \x n}$ be a matrix (we write $[n] := \{1, \dots, n\}$). We define a number of matrix norms. The most familiar one is the {\it spectral norm}:
\begin{equation} \label{eq:mat-spec}
\|A\| := \sup_{\substack{x \in \RR^n \\ |x| \le 1}} |Ax|
 = \sup_{\substack{x \in \RR^m, y \in \RR^n \\ |x|,|y| \le 1}} |x^*Ay|,
\end{equation}
where the second equality above follows from an application of the Cauchy--Schwarz inequality. Equivalently, $\|A\|$ equals the largest singular value of $A$. If $A$ is symmetric, then $\|A\|$ is also the maximum of the absolute values of the eigenvalues of $A$.

The notion of $\epsilon$-uniformity is captured by the {\it cut norm}:
\[
\|A\|_\cut := \sup_{S \subseteq [m], T \subseteq [n]} \left| \sum_{s \in S, t \in T} a_{st} \right|.
\]
It is not hard to see that the cut norm can be expressed as a linear relaxation:
\begin{equation} \label{eq:mat-cut-relax}
\|A\|_\cut = \sup_{x_1, \dots, x_m, y_1, \dots, y_n \in [0,1]} \left| \sum_{s \in [m], t \in [n]} a_{st} x_s y_t \right| = \sup_{x \in [0,1]^m, y \in [0,1]^n} |x^t A y|.
\end{equation}
To see the validity of this reformulation, note that the expression inside the absolute value is linear individually in each of $x_1, \dots, x_m, y_1, \dots, y_n$, and hence its extremum is attained when all these variables are $\{0,1\}$-valued, which is equivalent to the earlier formulation with subsets $S$ and $T$.

\medskip

We digress for a moment to relate these norms to graphs. Let $\Gamma$ be an $n$-vertex $d$-regular graph and let $A(\Gamma)$ be its adjacency matrix (recall that this is the $n \x n$ matrix $(a_{st})_{s, t \in [n]}$ where  $a_{st} = 1$ if $s$ and $t$ are adjacent in $\Gamma$ and $0$ otherwise). Then $\Gamma$ is an $(n,d,\lambda)$-graph if and only if 
\[
\| A(\Gamma) - \tfrac{d}{n} J \| \le \lambda,
\]
where $J$ is the $n \x n$ all-1's matrix, and $\Gamma$ is $\epsilon$-uniform if and only if 
\[
\|A(\Gamma) - \tfrac{d}{n} J \|_\cut \le \epsilon dn.
\]

\medskip

Returning to the discussion of matrix norms, we further relax the cut norm by allowing each $x_s$ and $y_t$ to be numbers in $[-1,1]$:
\begin{align*}
\|A\|_{\infty \to 1}
&:= \sup_{\substack{x \in \RR^n \\ |x|_\infty \le 1}} |Ax|_1
= \sup_{\substack{x \in \RR^m, y \in \RR^n \\ |x|_\infty,|y|_\infty \le 1}} |x^*Ay| \notag
\\
&= \sup_{x_1, \dots, x_m, y_1, \dots, y_n \in [-1,1]} \left| \sum_{s \in [m], t \in [n]} a_{st} x_s y_t \right|. \label{eq:infty-1-norm-r}
\end{align*}
This norm is equivalent to the cut norm:
\begin{equation}\label{eq:cut-infty-1-r}
  \|A\|_\cut \le \|A\|_{\infty \to 1} \le 4 \|A\|_\cut.
\end{equation}
Indeed, write $x = x_+ - x_-$ and $y = y_+ - y_-$, where $x_+,x_- \in [0,1]^m$ and $y_+,y_- \in [0,1]^n$, and then apply the triangle inequality and \eqref{eq:mat-cut-relax}.

Finally, we consider a semidefinite relaxation, which we shall refer to as the \emph{Grothendieck norm}:
\[
\|A\|_\G := \sup_{x_1, \dots, x_m, y_1, \dots, y_n \in B(\HH)} \left| \sum_{s \in [m], t \in [n]} a_{st} \ang{x_s, y_t} \right|,
\]
where $(\HH, \ang{ \cdot,\cdot})$ is any Hilbert space\footnote{In general, one should take a complex Hilbert space, but since we only consider real-valued matrices $A$, it is equivalent to use real Hilbert spaces.} and $B(\HH)$ is the unit ball in $\HH$, containing all points in $\HH$ of norm at most 1. In fact, since we are working with $m+n$ vectors, one can assume that $\HH = \RR^{m+n}$ in the definition above.

A key fact about the Grothendieck norm is that it is equivalent to the $\infty \to 1$ norm. Though this result was proved by Grothendieck~\cite{G53}, it was first stated in this form by Lindenstrauss and Pe\l czy\' nski~\cite{LP68}. We refer the interested reader to the surveys~\cite{KN12, P12} for further information.

\begin{theorem}[Grothendieck's inequality]
  \label{thm:grothendieck} There exists a constant $K_\G$ such that for all real-valued matrices $A$,
  \begin{equation*} \label{eq:grothendieck-r}
	\|A\|_{\infty \to 1} \le \|A\|_\G \le K_\G \|A\|_{\infty \to 1}.
  \end{equation*}
  We use the symbol $K_\G$ (the real Grothendieck constant) to denote the optimal constant in the above inequality.
\end{theorem}

The best current upper bound on the real Grothendieck constant is $K_\G < \frac{\pi}{2\log(1+\sqrt{2})} = 1.78\dots$, due to Braverman, Makarychev, Makarychev, and Naor~\cite{BMMN13} (see Haagerup~\cite{H87} for the best constant in the complex case). Combining \eqref{eq:cut-infty-1-r} with Grothendieck's inequality (and using $4K_\G \le 8$), we have
\begin{equation} \label{eq:mat-cut-G-r}
\|A\|_\cut \le \|A\|_\G \le 8 \|A\|_\cut,
\end{equation}
that is, the cut norm and the Grothendieck norm are equivalent.

\medskip

The spectral norm provides an upper bound on the Grothendieck norm: for any $A \in \RR^{m \x n}$,
\begin{equation} \label{eq:mat-G-spec}
\|A\|_\G \le \sqrt{mn} \|A\|.
\end{equation}
The proof is by a straightforward application of the Cauchy--Schwarz inequality.\footnote{We may assume that $\HH = \RR^k$ for some $k \le m+n$ and we have $x_1, \dots, x_m, y_1, \dots, y_n \in B(\HH)$ such that $\|A\|_\G = \bigl| \sum_{s \in [m], t \in [n]} a_{st} \ang{x_s, y_t} \bigr|
\le \sum_{j=1}^k  \bigl| \sum_{s \in [m], t \in [n]} a_{st} x_{sj} y_{tj} \bigr|
\le \sum_{j=1}^k \|A\| \sqrt{\sum_{s \in [m]} |x_{sj}|^2} \sqrt{ \sum_{t \in [n]} |y_{tj}|^2}
\le \|A\| \sqrt{\sum_{s \in [m]} |x_{s}|^2} \sqrt{ \sum_{t \in [n]} |y_{t}|^2}
\le \sqrt{mn} \|A\|$
by the triangle inequality, the definition of the spectral norm, and the Cauchy--Schwarz inequality.}

In general, the spectral norm and the Grothendieck norm are inequivalent, i.e., $\sqrt{mn}\|A\|/\|A\|_\G$ can be arbitrarily large, as shown by Example~\ref{nonequiv}. Our main theorem below shows, somewhat surprisingly, that the two norms agree for matrices associated to vertex-transitive graphs.

From now on we consider only square matrices. Let $A = (a_{s,t})_{1 \le s,t \le n}$. We say that a permutation $\sigma$ of $[n]$ is an \emph{automorphism} if $A_{\sigma(s),\sigma(t)} = A_{s,t}$ for all $s,t \in [n]$. We call a subgroup $G$ of the permutations of $[n]$ \emph{transitive} if for every $s,t \in [n]$ there is some $\sigma \in G$ such that $\sigma(s) = t$. We say that $A$ is \emph{vertex-transitive} if its group of automorphisms is transitive on $[n]$.

\begin{theorem} \label{thm:mat-G-spec}
$\|A\|_\G = n \|A\|$ for all $n \x n$ vertex-transitive real-valued matrices $A$.
\end{theorem}

Using \eqref{eq:mat-cut-G-r}, we obtain the following corollary relating the cut norm and the spectral norm, from which Corollary~\ref{transitive} clearly follows.

\begin{corollary} \label{cor:mat-spec-cut}
$\|A\|_\cut \le n\|A\| \le 8 \|A\|_\cut$  for all $n \x n$ vertex-transitive real-valued matrices $A$.
\end{corollary}

As discussed in the introduction, the proof will proceed through an analysis of Cayley graphs, which form the subject of the next section.

\section{Cayley graphs} \label{sec:cayley}

We consider a weighted version of Cayley graphs. Let $f \colon G \to \RR$. Define a weighted directed graph $\Gamma = \Cay(G,f)$ with vertex set $G$ and edge weights $\Gamma \colon G \x G \to \RR$ given by $\Gamma(g,h) := f(gh^{-1})$ for all $g,h \in G$. Let $A(f)$ denote the associated matrix, whose rows and columns are indexed by $G$, with $a_{g,h} = \Gamma(g,h) = f(gh^{-1})$ for all $g,h \in G$. The natural analogue of the assumption that $S$ is symmetric (giving undirected Cayley graphs) would be $f(g) = f(g^{-1})$, implying that $A(f)$ is symmetric, though we will not need to make such an assumption here.

We define various norms for functions $f \colon G \to \RR$ corresponding to the matrix norms defined in Section~\ref{sec:relax}. We use the averaging measure on $G$. Define the $p$-norms for functions on $G$ by
\[
\|f\|_p = \left(\EE_{g \in G} |f(g)|^p\right)^{1/p}
\]
and $\|f\|_\infty = \sup_{g \in G} |f(g)|$. Define the convolution $f_1 * f_2$ of two functions $f_1, f_2: G \rightarrow \mathbb{R}$ by 
\[(f_1 * f_2)(g) = \EE_{h \in G} f_1(gh^{-1}) f_2(h).\]
The spectral norm is defined as follows, agreeing with the matrix version (up to normalization):
\begin{equation} \label{eq:gp-spec}
\|f\| = \sup_{\substack{x \colon G \to \RR \\ \|x\|_2 \le 1}} \|f * x \|_2 
= \sup_{\substack{x,y \colon G \to \RR \\ \|x\|_2,\|y\|_2 \le 1}} \left|\EE_{g,h\in G} f(gh^{-1}) x(g) y(h) \right|
= |G|^{-1} \| A(f)\|.
\end{equation}
The Grothendieck norm is defined by
\[
\|f\|_\G := \sup_{x,y \colon G \to B(\HH)} \left| \EE_{g,h\in G} f(gh^{-1}) \ang{x(g),y(h)} \right| = |G|^{-2} \|A(f)\|_\G.
\]
Our main result now is that the spectral norm agrees with the Grothendieck norm for functions on a group.

\begin{theorem} \label{thm:gp-spec-G}
$\|f\| = \|f\|_\G$ for every function $f \colon G \to \RR$ on a finite group $G$.
\end{theorem}

We thank Assaf Naor and Prasad Raghavendra for independently suggesting the following short proof of Theorem~\ref{thm:gp-spec-G} after Y.~Z.\ presented an initial version of this work at the Simons Symposium on Analysis of Boolean Functions in April 2016. We include their proof here with permission. Our original proof using representation theory can be found in the appendix.

\medskip
\begin{proof}
Let $f \colon G \to \RR$. Let $x,y \colon G \to \RR$ with $\|x\|_2 \le 1$ and $\|y\|_2 \le 1$ be such that 
\[
\|f\| = |\EE_{g,h \in G} f(gh^{-1}) x(g) y(h)|.
\]
Define $x_g(h) = x(gh)$ and $y_g(h) = y(gh)$ for all $g,h \in G$. We view $x_g$ and $y_h$ as vectors in the unit ball in $L^2(G)$ equipped with inner product $\ang{x,y} = \EE_{g \in G} x(g)y(g)$ for all $x,y \in L^2(G)$. We have
\begin{align*}
\|f\|_\G 
&\ge |\EE_{g,h \in G} f(gh^{-1})\ang{x_g,y_h}|
= |\EE_{g,h,a \in G} f(gh^{-1}) x(ga)y(ha)|
\\&= |\EE_{g,h,a \in G} f((ga)(ha)^{-1}) x(ga)y(ha)|
= |\EE_{g,h \in G} f(gh^{-1}) x(g)y(h)|
= \|f\|.
\end{align*}
Combining with $\|f\|_\G \le \|f\|$ from \eqref{eq:mat-G-spec}, we obtain $\|f\|_\G = \|f\|$.
\end{proof}

\section{Transitive graphs}

Now we show that Theorem~\ref{thm:gp-spec-G} on functions $f \colon G \to \RR$ implies Theorem~\ref{thm:mat-G-spec} on vertex-transitive matrices. The idea is that we can lift the edges of a transitive graph to a Cayley graph over its automorphism group and this operation preserves all the norms that we are interested in.

Recall that $A = (a_{s,t})_{s,t \in [n]}$ is an $n \times n$ vertex-transitive matrix if the group of automorphisms of $A$ is transitive on $[n]$. Let $G$ be any transitive subgroup of automorphisms (it could be the whole automorphism group or any transitive subgroup). Here every element $g \in G \le S_n$ is some permutation of $[n]$, i.e., $g \colon [n] \to [n]$, and the product of two elements $g,h \in S_n$ is defined by $(gh)(s) = h(g(s))$ for all $s \in [n]$. We have $a_{g(s),g(t)} = a_{s,t}$ for all $s,t \in [n]$ and $g \in G$.

\begin{lemma} \label{lem:uplift}
  Let $A = (a_{s,t})_{s,t \in [n]}$ be an $n \times n$ vertex-transitive real-valued matrix and let $G$ be any transitive subgroup of automorphisms of $A$. Define $f \colon G \to \RR$ by $f(g) = a_{g(1), 1}$. Then
\[
n \|f\| = \|A\| \quad
\text{and} \quad
n^2 \|f\|_\G = \|A\|_\G.
\]
\end{lemma}

\begin{proof}
The proofs of these two identities are essentially the same. Let $\HH$ denote either an arbitrary Hilbert space (for the Grothendieck norm) or simply $\RR$ (for the spectral norm). For any $x_1, \dots, x_n, y_1, \dots, y_n \in \HH$, define $x,y \colon G \to \HH$ by $x(g) = x_{g(1)}$ and $y(g) = y_{g(1)}$ for all $g \in G$. Then
\begin{multline}
  \label{eq:trans-lift-form}
  \EE_{g,h \in G} [f(gh^{-1}) \langle x(g), y(h) \rangle]
= \EE_{g,h \in G} [a_{(gh^{-1})(1), 1} \langle x(g), y(h) \rangle]\\
= \EE_{g,h \in G} [a_{g(1), h(1)} \langle x_{g(1)}, y_{h(1)} \rangle]
= \frac{1}{n^2} \sum_{s,t \in [n]} a_{s,t} \langle x_s, y_t \rangle.
\end{multline}
Furthermore, we have $n^{-1} \sum_{s \in [n]} |x_s|^2 = \EE_{g \in G} |x(g)|^2$, and similarly with $y$. This shows $n\|f\| \ge \|A\|$ and $n^2\|f\|_\G \ge \|A\|_\G$.

Conversely, given $x,y \colon G \to \HH$, we can set $x_1, \dots, x_n, y_1, \dots, y_n \in \HH$ to be $x_s = \EE_{g \in G : g(1) = s} x(g)$ and $y_s = \EE_{g \in G : g(1) = s} y(g)$ for all $s \in [n]$. Then \eqref{eq:trans-lift-form} still holds. We also have $n^{-1} \sum_{s \in [n]} |x_s|^2 \le \EE_{g \in G} |x(g)|^2$ by convexity, and similarly with $y$. This shows that $n\|f\| \le \|A\|$ and $n^2\|f\|_\G \le \|A\|_\G$.
\end{proof}

\section{Bipartite analogues}

There are also variants of our results in the bipartite setting. Given a group $G$ of order $n$ and a function $f \colon G \rightarrow \RR$, the \emph{bipartite Cayley graph} $\Gamma = \Bip(G, f)$ is the bipartite graph between two copies of $G$ with edge weights $f \colon G \x G \to \RR$ given by $\Gamma(g, h) = f(gh^{-1})$. We write $B(f)$ for the associated matrix, whose rows and columns are both indexed by $G$, with $b_{g,h} = \Gamma(g,h) = f(g h^{-1})$. Since we did not require that $f$ be symmetric in Theorem~\ref{thm:gp-spec-G}, we can apply it directly to prove that $\|B(f)\|_\G = n \|B(f)\|$ for any $f \colon G \rightarrow \RR$. An analogue of Theorem~\ref{thm:cayley} then follows provided we replace the second largest eigenvalue with the largest singular value of $B(\Gamma) - \frac{d}{n} J$, where $J$ is the all-1's matrix.

We also have an analogue of Theorem~\ref{transitive}. We say that a bipartite graph $\Gamma$ between sets $U$ and $V$ of orders $m$ and $n$, respectively, is {\it bitransitive} if the group of automorphisms of $\Gamma$ is transitive on each part. For example, the graph between any two consecutive layers of the Boolean cube satisfies this condition. A slight amendment to the proof of Lemma~\ref{lem:uplift} then allows us to show that if $B$ is an $m \times n$ bitransitive real-valued matrix, then $\|B\|_\G = \sqrt{mn} \|B\|$. The claimed analogue of Theorem~\ref{transitive} then follows easily, with the second largest eigenvalue now replaced by the largest singular value of $B(\Gamma) - p J$, where $p$ is the density of $\Gamma$.

\medskip
\noindent\emph{Acknowledgments.} We would like to thank Sean Eberhard for helpful discussions. We would also like to thank Noga Alon and Mathias Schacht for bringing the paper~\cite{ACHKRS10} to our attention. We are greatly indebted to Assaf Naor and Prasad Raghavendra for their contribution and Y.~Z.\ would like to thank the Simons Foundation and the organisers of the Simons Symposium on Analysis of Boolean Functions for facilitating these conversations. Finally, we thank the referees for suggesting changes that greatly improved the presentation of the paper.

\appendix

\section{Representation-theoretic proof}

We include in this appendix our original proof of Theorem~\ref{thm:gp-spec-G}, that $\|f\| = \|f\|_\G$ for $f \colon G \to \RR$, using representation theory. The proof follows the spirit of Gowers' work on quasirandom groups~\cite{G08} and is a  natural generalization of the proof for the abelian case \cite{KRS16}.

\subsection{Abelian groups}
As a warm up, we give a short proof in the abelian setting, which is significantly easier. This proof is instructive, though not strictly required for the general setting.

Here and throughout the appendix, we will work with functions $f \colon G \to \CC$ and with complex analogues of our various norms. For the most part, the definitions of these norms are the same as in the real case or the same after some minor modification. For example, letting $\DD := \{z \in \CC : |z| \le 1\}$ denote the unit disk in the complex plane, the complex $\infty \to 1$ norm for functions $f \colon G \to \CC$ is defined by
\[
\|f\|^\CC_{\infty \to 1} := \sup_{x,y \colon G \to \DD} \left| \EE_{g,h \in G} f(gh^{-1}) \ol{x(g)} y(h) \right|.
\]
The essence of the proof below can be found in Kohayakawa--R\"odl--Schacht~\cite{KRS16} (where the proof is attributed to Gowers). Here we give a more streamlined presentation.

\begin{theorem} \label{thm:abelian}  $\|f\| = \|f\|_{\infty \to 1}^\CC$ for every function $f \colon G \to \CC$ on a finite abelian group $G$.
\end{theorem}

\begin{proof}
It is easy to see that $\|f\|_{\infty \to 1}^\CC \le \|f\|$ from the definitions, as $\|x\|_2 \le 1$ for all $x \colon G \to \DD$.
It is well known that $\|f\| = |\EE_{g \in G} f(g)\ol{\chi(g)}|$ for some multiplicative character $\chi \colon G \to \CC$. On the other hand, since $\|\chi\|_\infty = 1$, we have 
\[
\|f\|_{\infty \to 1}^\CC \ge |\EE_{g,h \in G} f(gh^{-1})\ol{\chi(g)}\chi(h)| = |\EE_{g \in G} f(g)\ol{\chi(g)}| = \|f\|.
\]
Therefore, $\|f\| = \|f\|_{\infty \to 1}^\CC$.
\end{proof}

Similarly to \eqref{eq:cut-infty-1-r}, it can be shown that the cut norm is equivalent to the complex $\infty \to 1$ norm for complex-valued matrices. Together with Theorem~\ref{thm:abelian}, this implies that the cut norm and the spectral norm are equivalent for abelian Cayley graphs.

\begin{remark}
For non-abelian groups $G$, it is possible to have $\|f\| > \|f\|_{\infty \to 1}^\CC$.
If an undirected weighted Cayley graph on $G$ admits a coordinate-wise uniformly bounded eigenbasis, then one has $\|f\| = O(\|f\|_{\infty \to 1}^\CC)$ by the above argument.\footnote{We thank Assaf Naor for this remark.} However, as will be addressed in future work, not all Cayley graphs admit such eigenbases. Consequently, the above proof does not extend immediately to the non-abelian case. The proof given below instead generalizes the above proof by replacing the multiplicative characters with group representations.
\end{remark}

In the next three subsections, we recall some standard mathematical tools used in our proof. The reader familiar with these tools should feel free to skip ahead to Section~\ref{sec:proof-main}.

\subsection{Group representation theory} We recall some basic facts from group representation theory. They can be found in any standard textbook on the subject, e.g.,~\cite{S77}.

Let $G$ be a finite group, not necessarily abelian. We write $\wh G$ for the set of distinct irreducible representations of $G$. An irreducible representation $\rho \in \wh G$ of dimension $d_\rho$ is a homomorphism $\rho \colon G \to U(d_\rho)$, where $U(n)$ is the group of unitary $n \x n$ matrices (so that $\rho(g^{-1}) = \rho(g)^*$). Note that by a standard averaging trick, we need to only consider unitary representations.

We will need the following orthogonality result known as Schur's lemma.

\begin{theorem}[Schur's lemma] \label{thm:schur}  Let $G$ be a finite group and let $\rho, \sigma \in \wh G$. Let $M$ be any $d_\rho \x d_\sigma$ matrix with complex entries. Then
\[
\EE_{g \in G} [\rho(g) M \sigma(g^{-1})] =
\begin{cases}
\frac{\Tr M}{d_\rho}  I & \text{if } \rho = \sigma, \\
0 & \text{otherwise}.
\end{cases}
\]
\end{theorem}

\subsection{Non-abelian Fourier analysis}
Let us first recall the Fourier transform for abelian groups. Let $G$ be a finite abelian group and $\wh G$ the group of multiplicative characters $\chi \colon G \to \CC$. For any function $f \colon G \to \CC$, we can define its Fourier transform $\wh f \colon \wh G \to \CC$ by $\wh f(\chi) = \EE_{g \in G} f(g)\chi(g)$. We then have the inversion formula $f(g) = \sum_{\chi \in \wh G} \wh f(\chi) \ol{\chi(g)}$.

From now on, let $G$ be any finite group, not necessarily abelian, and $\wh G$ the set of distinct irreducible unitary representations $\rho \colon G \to U(d_\rho)$ of $G$. For any $f \colon G \to \CC$ and any irreducible representation $\rho \in \wh G$, define the Fourier transform of $f$ at $\rho$ by
\begin{equation*} \label{eq:fourier}
\wh f (\rho) = \EE_{g \in G} f(g) \rho(g).
\end{equation*}
Note that $\wh f(\rho)$ is a $d_\rho \x d_\rho$ matrix, so its dimension varies with $\rho$. 

For any two complex-valued matrices $A = (a_{st})_{s \in [m],t \in [n]}$ and $B = (b_{st})_{s \in [m],t \in [n]}$ with the same dimensions, define the Hilbert--Schmidt inner product
\[
\ang{A,B}_\HS = \Tr(AB^*) = \sum_{s \in [m], t \in [n]} a_{st} \ol{b_{st}}
\]
and the Hilbert--Schmidt norm
\[
\|A\|_\HS = \sqrt{\ang{A,A}_\HS} = \sqrt{ \sum_{s \in [m], t \in [n]} |a_{st}|^2}.
\]

We have a Fourier inversion formula:
\begin{equation} \label{eq:fourier-inv}
f(g) = \sum_{\rho \in \wh G} d_\rho \langle \wh f(\rho),\rho(g)\rangle_{\HS}
= \sum_{\rho \in \wh G} d_\rho \Tr(\wh f(\rho) \rho(g)^*).
\end{equation}
A proof of this formula can be found in~\cite[Section 2.10]{W10}.
We also have Plancherel's identity: for all $f_1, f_2 \colon G \to \CC$,
\begin{equation*} \label{eq:par-inner}
	\EE_{g \in G} f_1(g)\ol{f_2(g)} 
	= \sum_{\rho \in \wh G} d_\rho \langle \wh{f_1}(\rho), \wh{f_2}(\rho)\rangle_\HS.
\end{equation*}
In particular, setting $f_1 = f_2 = f$, we have that for all $f \colon G \to \CC$,
\begin{equation} \label{eq:par-norm}
	\|f\|_2^2 = \EE_{g \in G} \left[|f(g)|^2 \right]
	= \sum_{\rho \in \wh G} d_\rho \| \wh f(\rho)\|_\HS^2.
\end{equation}
The Fourier transform converts convolution into multiplication. For all $f_1, f_2 \colon G \to \CC$, writing
\[
f_1 * f_2 (g) := \EE_{h \in G} f_1(gh^{-1}) f_2(h)
\]
for convolution, we have, for all $\rho \in \wh G$,
\begin{equation} \label{eq:fourier-conv}
\wh{f_1 * f_2} (\rho) = \EE_{g,h \in G} f_1(g)f_2(h)\rho(gh)
= \left(\EE_{g\in G} f_1(g)\rho(g)\right)\left(\EE_{h\in G} f_2(h)\rho(h)\right)
= \wh{f_1} (\rho) \wh{f_2} (\rho).
\end{equation}
Or, more succinctly, $\wh{f_1 * f_2} = \wh{f_1}\wh{f_2}$.

Conceptually, the Fourier transform is a simultaneous block diagonalization of all $f \colon G \to \CC$ (viewed as linear maps on the space of functions $x \colon G \to \CC$ via the convolution $x \mapsto f * x$). For each $\rho \in \wh G$, the $d_\rho \x d_\rho$ matrix $\wh f(\rho)$ appears exactly $d_\rho$ times in the block diagonalization (the multiplicity $d_\rho$ is the multiplicity of $\rho$ in the regular representation of $G$). As a result, the spectral norm of $f$ must be the maximum of the spectral norms of the components of the block diagonalization:
\begin{equation}
  \label{eq:fourier-norm}
  \|f\| = \max_{\rho \in \wh G} \|\wh f(\rho) \|,
\end{equation}
where on the left-hand side $\|f\|$ is the spectral norm \eqref{eq:gp-spec} of a function $f \colon G \to \CC$ and on the right-hand side $\|\wh f(\rho)\|$ is the spectral norm~\eqref{eq:mat-spec} of a matrix. Let us also give a direct proof of \eqref{eq:fourier-norm} for the convenience of the reader. Using \eqref{eq:fourier-conv} and Plancherel's identity \eqref{eq:par-norm}, we have
\begin{equation} \label{eq:fourier-norm-pf1}
\|f\|^2 
= \sup_{\substack{x \colon G \to \CC \\ \|x\|_2 \le 1}} \|f * x \|_2^2
= \sup_{\substack{x \colon G \to \CC \\ \|x\|_2 \le 1}} \sum_{\rho \in \wh G} d_\rho \|\wh f(\rho) \wh x(\rho) \|_\HS^2.
\end{equation}
If $A$ and $B$ are $d \x d$ matrices and $b_1, \dots, b_d$ are the columns of $B$, then $\|AB\|_\HS^2 = |Ab_1|^2 + \cdots + |Ab_d|^2 \le \|A\|^2(|b_1|^2 + \cdots |b_d|^2) = \|A\|^2 \|B\|_\HS^2$. Therefore, for any $x \colon G \to \CC$, we have
\begin{align*}
\sum_{\rho \in \wh G} d_\rho \|\wh f(\rho) \wh x(\rho) \|_\HS^2
&\le \sum_{\rho \in \wh G} d_\rho \|\wh f(\rho)\|^2 \|\wh x(\rho) \|_\HS^2 
\\
&\le \bigl(\max_{\rho \in \wh G} \|\wh f(\rho) \|\bigr)^2 \sum_{\rho \in \wh G} d_\rho \|\wh x(\rho) \|_\HS^2
= \bigl(\max_{\rho \in \wh G} \|\wh f(\rho) \|\bigr)^2  \|x\|_2^2,
\end{align*}
where the final step uses Plancherel \eqref{eq:par-norm} again. Combining with \eqref{eq:fourier-norm-pf1}, this shows that $\|f\| \le \max_{\rho \in \wh G} \|\wh f(\rho)\|$. Conversely, for a fixed $\rho$, by setting the columns of $\wh x(\rho)$ to be the top right singular vector of $\wh f(\rho)$ so that $\|\wh f(\rho)\wh x (\rho)\|_\HS = \|\wh f(\rho)\| \|\wh x (\rho)\|_\HS$, and setting $\wh x (\sigma) = 0$ for all $\sigma \ne \rho$, we obtain $\|f\| \ge \| \wh f(\rho)\|$ for each $\rho \in \wh G$. This proves \eqref{eq:fourier-norm}.

\subsection{Singular value decompositions}
Every $m \times n$ complex matrix $M$ of rank $r$ has a singular value decomposition (SVD):
\[
M = \lambda_1 u_1v_1^* + \dots + \lambda_r u_r v_r^*,
\]
where $\lambda_1 \ge \cdots \ge \lambda_r > 0$ (known as the singular values), $u_1, \dots, u_r$ are mutually orthogonal unit vectors in $\CC^m$, and $v_1, \dots, v_r$ are mutually orthogonal unit vectors in $\CC^n$.

Let $f \colon G \to \CC$. For each $\rho \in \wh G$, take any singular value decomposition of $\wh f(\rho)$:
\[
\wh f(\rho) = \lambda_1^\rho u_1^\rho v_1^{\rho*} + \cdots + \lambda_{d_\rho}^\rho u_{d_\rho}^\rho v_{d_\rho}^{\rho*},
\]
where $\lambda_1^\rho \ge \dots \ge \lambda_{d_\rho}^\rho \ge 0$ (we allow some of the $\lambda_j^\rho$'s to be zero), and $\{u_1^\rho, \dots, u_{d_\rho}^\rho\}$ and  $\{v_1^\rho, \dots, v_{d_\rho}^\rho\}$ are both orthonormal bases of $\CC^{d_\rho}$. For any $\lambda \ge 0$, $u,v \in \CC^{d_\rho}$, we have $\Tr(\lambda u v^* \rho(g)^*) = \lambda  v^* \rho(g)^* u$. Therefore, by the Fourier inversion formula \eqref{eq:fourier-inv}, we have the following SVD for any $f \colon G \to \CC$:
\begin{equation}
  \label{eq:f-svd}
  f(g) = \sum_{\rho \in \wh G} d_\rho \sum_{k=1}^{d_\rho} \lambda_k^\rho v_k^{\rho*} \rho(g)^* u_k^\rho.
\end{equation}

\subsection{Proof of Theorem~\ref{thm:gp-spec-G}}\label{sec:proof-main}

Now we give the representation-theoretic proof that $\|f\| = \|f\|_\G$ for any $f \colon G \to \CC$.

Consider the SVD \eqref{eq:f-svd} of $f$. From \eqref{eq:fourier-norm}, we see that $\|f\|$ is the maximum of the top singular value $\lambda_1^\rho$ ranging over all $\rho \in \wh G$. Let $\sigma \in \wh G$ be such that $\|f\| = \lambda_1^\sigma$.

We already know from \eqref{eq:mat-G-spec} that $\|f\|_\G \le \|f\|$. To show $\|f\|_\G \ge \|f\|$, it suffices to exhibit $x,y \colon G \to B(\HH)$ so that $\EE_{g,h} f(gh^{-1}) \ang{x(g),y(h)} = \|f\|$ (here $\ang{x,y} := \ol{x_1}y_1 + \cdots + \ol{x_n}y_n$ for $x,y \in \CC^n$). Set $\HH = \CC^{d_\sigma}$ and, for all $g,h \in G$,
\[
x(g) = \ol{\sigma(g^{-1}) v_1^\sigma} \quad \text{and} \quad y(h) = \ol{\sigma(h^{-1}) u_1^\sigma},
\]
where $u_1^\sigma$ and $v_1^\sigma$ are singular vectors from the SVD \eqref{eq:f-svd}. Let $\rho \in \wh G$ and $1 \le k \le d_\rho$. Writing $u = u_k^\rho$ and $v = v_k^\rho$ to reduce clutter in notation, we have
\begin{align*}
\EE_{g,h}d_\rho  v^*\rho (gh^{-1}) u \ang{x(g),y(h)}
&=
\EE_{g,h}d_\rho v^*\rho (gh^{-1}) u \ol{\ang{y(h),x(g)}}
\\
&=
\EE_{g,h}d_\rho v^*\rho (gh^{-1}) u \ang{\sigma(h^{-1}) u_1^\sigma,\sigma(g^{-1}) v_1^\sigma}
\\
&=
\EE_{g,h}d_\rho v^*\rho (gh^{-1}) u u_1^{\sigma*}\sigma(hg^{-1}) v_1^\sigma
\\
&=
\EE_{g}d_\rho v^*\rho (g) u u_1^{\sigma*}\sigma(g^{-1}) v_1^\sigma.
\end{align*}
By Schur's lemma (Theorem~\ref{thm:schur}), if $\rho \ne \sigma$, then this is zero. If $\rho = \sigma$, then it equals
\begin{align*}
  v^* \Tr(u u_1^{\sigma*})v_1^\sigma = \ang{u_1^\sigma,u}\ang{v,v_1^\sigma},
\end{align*}
which is 1 if $k = 1$, i.e., $u = u_1^\sigma$ and $v = v_1^\sigma$, and zero otherwise due to the orthogonality of the singular vectors. It follows, by the SVD \eqref{eq:f-svd} and the above calculation, that
\[
\EE_{g,h} f(gh^{-1}) \ang{x(g),y(h)} = \lambda_1^\sigma = \|f\|.
\]
This shows that $\|f\|_\G \ge \|f\|$, completing the proof. \qed

\begin{remark}
  The above proof shows that in defining $\|f\|_\G$ for $f \colon G \to \CC$, one only needs to take $\HH$ (in the definition of the Grothendieck norm) to have dimension at most $\max_{\rho \in \wh G} d_\rho$, the maximum dimension of an irreducible representation of $G$. In other words, if we define intermediate Grothendieck-type norms
\[
\|f\|_{\G,k} := \sup_{x,y \colon G \to B(\CC^k)} \left| \EE_{g,h\in G} f(gh^{-1}) \ang{x(g),y(h)} \right|
\]
for every positive integer $k$, so that
\[
\|f\|_{\infty \to 1} = \|f\|_{\G,1} \le \|f\|_{\G,2} \le \cdots \le \|f\|_{\G,m+n}  = \|f\|_\G,
\]
then $\|f\|_\G = \|f\|_{\G,k}$ whenever all irreducible representations of $G$ have dimension at most $k$. In particular, if $G$ is abelian, then all irreducible representations have dimension 1, so $\|f\| = \|f\|_{\G,1} = \|f\|^\CC_{\infty\to 1}$, as shown earlier in Theorem~\ref{thm:abelian}.

\end{remark}

\begin{dajauthors}
\begin{authorinfo}[dc]
  David Conlon\\
  Mathematical Institute\\
  University of Oxford\\
  Oxford OX2 6GG, United Kingdom \\
  david\imagedot{}conlon\imageat{}maths\imagedot{}ox\imagedot{}ac\imagedot{}uk
\\
  \url{https://www.maths.ox.ac.uk/people/david.conlon}
\end{authorinfo}
\begin{authorinfo}[yz]
  Yufei Zhao\\
  Mathematical Institute\\
  University of Oxford\\
  Oxford OX2 6GG, United Kingdom \\
  yufei\imagedot{}zhao\imageat{}maths\imagedot{}ox\imagedot{}ac\imagedot{}uk
\\
  \url{https://yufeizhao.com}
\end{authorinfo}
\end{dajauthors}

\end{document}